\documentclass{article}
\usepackage{mytemplate}
\usepackage[alphabetic]{amsrefs}

\newtheorem{thm}{Theorem}
\newtheorem{prop}{Proposition}[section]
\newtheorem{cor}[prop]{Corollary}
\newtheorem{lem}[prop]{Lemma}

\renewcommand{\a}{\mathbf{a}}
\renewcommand{\b}{\mathbf{b}}
\renewcommand{\A}{\mathbf{A}}

\newcommand{\irreg}{\mathrm{irreg}}
\newcommand{\one}{\mathbf{1}}
\renewcommand{\hat}{\widehat}
\newcommand{\fqb}{{\overline{\F}_q}}
\newcommand{\comp}{\mathrm{comp}}

\begin{document}

\title{Factorization Statistics of Restricted Polynomial Specializations over Large Finite Fields}
\author{Alexei Entin}
\date{}
\maketitle

\begin{abstract} For a polynomial $F(t,A_1,\ldots,A_n)\in\F_p[t,A_1,\ldots,A_n]$ ($p$ being a prime number) we study the factorization statistics of its specializations $$F(t,a_1,\ldots,a_n)\in\F_p[t]$$ with $(a_1,\ldots,a_n)\in S$, where $S\ss\F_p^n$ is a subset, in the limit $p\to\ity$ and $\deg F$ fixed. We show that for a sufficiently large and regular subset $S\ss\F_p^n$, e.g. a product of $n$ intervals of length $H_1,\ldots,H_n$ with $\prod_{i=1}^nH_n>p^{n-1/2+\eps}$, the factorization statistics is the same as for unrestricted specializations (i.e. $S=\F_p^n$) up to a small error. This is a generalization of the well-known P\'olya-Vinogradov estimate of the number of quadratic residues modulo $p$ in an interval.
\end{abstract}

\section{Introduction}

Let $p$ be a prime. A well-known result of P\'olya and Vinogradov \cite{Pol18}\cite{Vin18} states that the number of quadratic residues (as well as the number of nonresidues) modulo $p$ in the interval $I=\{\be,\be+1,\ldots,\be+H-1\}$ of length $H$ is $H/2+O(p^{1/2}\log p)$. . This result can be restated as follows: the number of $a\in I$ such that the polynomial $t^2-a$ is reducible (or irreducible) modulo $p$ is $H/2+O(p^{1/2}\log p)$. Vinogradov extended his result to higher power residues \cite{Vin27}, he has shown that the number of $a\in I$ such that the polynomial $t^k-a$ has a root modulo $p$ is
$H/(p-1,k)+O(p^{1/2}\log p)$. Our aim is to generalize these results to specializations of any polynomial in $n$ variables over any finite field.

Shparlinski \cite{Shp89}\cite{Shp92} has studied the following related problem: for $n\ge 2$ consider the set of polynomials $$\mathcal{S}=\{f=t^n+a_1t^{n-1}+\ldots+a_n:a_i\in I_i\},$$ $$I_i=\{\be_i,\ldots,\be_i+H_i-1\},\be_i\in\F_p,H_i\le p$$
and let $d_1,\ldots,d_r$ be natural numbers with $\sum d_i=n$. We say that a squarefree polynomial $f\in\F_p[t],\deg f=n$ has factorization type $(d_1,\ldots,d_n)$ if one can write $f=\prod_{i=1}^r P_i$ with $P_i\in\F_p[t]$ irreducible with $\deg P_i=d_i$ (in particular $f$ is irreducible iff it has factorization type $(n)$). Denote by $D(d_1,\ldots,d_r)$ the set of monic squarefree $f\in\F_p[t]$ with factorization type $(d_1,\ldots,d_r)$. Shparlinski \cite{Shp92}*{Theorem 3} has shown that in the limit $p\to\ity$ the number of $f\in\mathcal{S}$ with factorization type $(d_1,\ldots,d_r)$ is
\begin{equation}|\mathcal{S}\cap D(d_1,\ldots,d_r)|=\label{shpa1}\gam(d_1,\ldots,d_r)\prod_{i=1}^nH_i+O_{n}\lb p^{n-1}\log^n p\rb,\end{equation}
where $\gam(d_1,\ldots,d_r)$ is the probability that a random permutation in $S_n$ has cycle structure $(d_1,\ldots,d_r)$ and the notation $O_{n}$ means that the implicit constant may depend on $n$.
In \cite{Shp11} he obtained a similar result for the set of trinomials with coefficients lying in a rectangle
$$\mathcal{S}=\{f=t^n+a_1t+a_2:a_1\in I_1,a_2\in I_2\}, I_i=\{\be_i,\be_i+H_i-1\},\be_i\in\F_p, H_i\le p.$$ He has shown that for any fixed $\eps>0$ if $H_1,H_2\ge p^{1/4+\eps}$ and $H_1H_2\ge p^{1+\eps}$ the number of $f\in\mathcal{S}$ with factorization type $(d_1,\ldots,d_r)$ is asymptotically 
\begin{equation}\label{shpa2}|\mathcal{S}\cap D(d_1,\ldots,d_r)|\sim\gam(d_1,\ldots,d_r)H_1H_2.\end{equation}
 Our main results will imply as special cases weaker versions of the above results with an error term of $O_{n}\lb p^{n-1/2}\log^n p\rb$ in
(\ref{shpa1}) and the condition $H_1H_2\ge p^{3/2+\eps}$ for (\ref{shpa2}). While we obtain weaker results in these special cases, our results apply to polynomials of a far more general shape. Some further interesting applications to the distribution of irreducible polynomials will be described in the next section.

We now describe the general problem we want to consider, which generalizes all the problems discussed above. With $p$ still denoting a prime let $q=p^k$ be a power of $p$ and $\F_q$ the field with $q$ elements. Let $F(t,A_1,\ldots,A_n)\in\F_q[t,A_1,\ldots,A_n]$ be a squarefree polynomial without irreducible factors lying in $\F_q[t^p,A_1,\ldots,A_n]$. These two conditions are equivalent to the single condition $\Disc_tF\neq 0$ ($\Disc_tF$ denotes the discriminant of $F$ viewed as a polynomial in the variable $t$. This discriminant is a polynomial in $A_1,\ldots,A_n$). Denote $d=\deg_t F$. We fix the total degree $\deg F$ and consider the limit $q\to\ity$. We consider specializations $F(t,a_1,\ldots,a_n)\in\F_q[t]$ with $a_i\in\F_q$. To simplify notation we denote $$\A=(A_1,\ldots,A_n),\a=(a_1,\ldots,a_n).$$ The condition $\Disc_t F\neq 0$ insures that all but $O_{n,\deg F}(q^{n-1})$ of these specializations are squarefree of degree $d$ (the notation $O_{n,\deg F}$ means that the implicit constant depends only on $n,\deg F$) and we may consider the prime factorization $F(t,\a)=\prod_{i=1}^dP_i(t)$ with $\deg P_i=d_i$. For such an $\a$ denote by $C(F(t,\a))$ the conjugacy class in $S_d$ of permutations with cycle structure $(d_1,\ldots,d_r)$. We call it the factorization class of $F(t,\a)$.

It is known that for any fixed conjugacy class $C$ of $S_d$ the probability that $C(F(t,\a))=C$ as $\a$ ranges over $\F_q^n$ is determined by the Galois group $G$ of the polynomial $F(t,\A)$ over the field $\F_q(\A)$ together with its standard action on the roots,
up to an error term of $O_{n,\deg F}(q^{-1/2})$. More precisely if $\F_{q^\nu}$ is the algebraic closure of $\F_q$ in the splitting field of $F(t,\A)$ over $\F_q(\A)$ and $\pi:G\to\Gal\lb\F_{q^\nu}/\F_q\rb$ is the restriction map, then for any $S_d$-conjugacy class $C$ we have
$$\card{\{\a\in\F_q^n|C(F(t,\a))=C\}}=\frac{\nu\card{C\cap\pi^{-1}(\Fr_q)}}{|G|}q^n\lb 1+O_{n,\deg F}\lb q^{-1/2}\rb\rb,$$ where $\Fr_q$ denotes the $q$-Frobenius automorphism of $\F_{q^\nu}$ (see Corollary \ref{corcheb} below).
This follows from an explicit Chebotarev density theorem for varieties over finite fields, see \cite{Ent18_}*{Theorem 3} or \cite{ABR15}*{Theorem A.4}.

For many interesting examples of polynomials $F(t,\A)$ this Galois group has been computed. See \cite{Ent18_} for a fairly general result of this type and further references to the literature and \cite{BBF18},\cite{Kar16_} for further examples.

Now suppose that we take a subset $S\ss\F_q^n$ and only allow specializations $\a\in S$. We are going to show that if $S$ is sufficiently large and regular in a sense we define shortly the distribution of $C(F(t,\a))$ remains the same up to a small error term as $q\to \ity$ and $n,\deg F$ are fixed, more precisely we would have
\begin{multline*}\card{\{\a\in S|C(F(t,\a))=C\}}=\\=\frac{\nu\card{C\cap\pi^{-1}\lb\Fr_q\rb}}{|G|}|S|\lb 1+O_{n,\deg F}\lb q^{-1/2}\cdot\irreg(S)\rb\rb,\end{multline*} where $\irreg(S)$ is a measure of the size and irregularity of $S$ that we define next.

Let $f:\F_q^n\to\C$ ($\C$ denotes the field of complex numbers) be a function. We define its Fourier transform by
\begin{equation}\label{fourier}\hat{f}(\b)=\frac{1}{q^n}\sum_{\a\in\F_q^n}f(\a)e^{-2\pi i\frac{\mathrm{tr}_{\F_q/\F_p}(\a\cdot\b)}{p}}\end{equation} (for $\a=(a_1,\ldots,a_n),\b=(b_1,\ldots,b_n)$ we denote $\a\cdot\b=a_1b_1+\ldots+a_nb_n$). For a subset $S\ss\F_q^n$ we define its \emph{irregularity} to be
\begin{equation}\label{irreg}\irreg(S)=\frac{q^n}{|S|}\sum_{\b\in\F_q^n}\abs{\hat{\one_S}(\b)},\end{equation} where $\one_S$ denotes the indicator function of $S$.
For example if $q=p$ is a prime and $S=I_1\times\ldots\times I_n$ where $I_i\ss\F_p$ are intervals or arithmetic progressions of length $H_i$ then 
\begin{equation}\label{irregint}\irreg(S)\le \frac {(9p\log p)^n}{\prod_{i=1}^nH_i}\end{equation} (see Appendix). Our main result is the following

\begin{thm}\label{thm1} Let $S\ss\F_q^n$ be a subset, $F(t,\A)\in\F_q[t,\A]$ a squarefree polynomial without irreducible factors in $\F_q[t^p,\A]$ (equivalently $\Disc_tF\neq 0$) and $\deg_t F=d$. Assume that $p>d$. Let $G\ss S_d$ be its Galois group over $\F_q(\A)$ with the standard permutation action on the roots, $\F_{q^\nu}$ the algebraic closure of $\F_q$ in its splitting field and $\pi:G\to\Gal\lb\F_{q^\nu}/\F_q\rb$ the restriction map. Let $C$ be a conjugacy class of $S_d$. Then
\begin{multline*}\card{\{\a\in S|C(F(t,\a))=C\}}=\\=\frac{\nu\card{C\cap\pi^{-1}\lb \Fr_q\rb}}{|G|}|S|\lb 1+O_{n,\deg F}\lb q^{-1/2}\cdot\irreg(S)\rb\rb,\end{multline*} the implicit constant only depending on $n,\deg F$.\end{thm}

\begin{cor}\label{corint} In the setting of Theorem \ref{thm1} assume that $q=p$ is a prime, $n,d,\eps>0$ are fixed and $S=I_1\times\ldots\times I_n$ is a product of intervals or arithmetic progressions of size $|I_i|=H_i$ with $\prod_{i=1}^nH_i\ge p^{n-1/2+\eps}$. Then as $p\to\ity$ the classes $C(F(t,\a)),\a\in S$ become equidistributed with respect to the measure assigning to each conjugacy class $C$ of $S_d$ the size of $C\cap\pi^{-1}\lb\Fr_q\rb$.\end{cor}

\begin{proof} Follows from Theorem \ref{thm1} and (\ref{irregint}).\end{proof}

{\bf Remark 1.} The case $n=1,d=2$ and $F(t,A)=t^2-A$ in the above corollary is the result of Polya and Vinogradov on the equidistribution of quadratic residues and nonresidues in intervals of length $p^{1/2+\eps}$. Indeed $C(F(t,a))$ is the identity class iff $a$ is a nonzero quadratic residue and the class of 2-cycles iff $a$ is a nonresidue. For this special case the result was extended to intervals of length $p^{1/4+\eps}$ by Burgess \cite{Bur57}, and this result remains the state of the art. Conjecturally it is still true for intervals of length $p^\eps$ and we conjecture that Corollary \ref{corint} also holds whenever $H_i\ge p^\eps$, but this remains out of reach even for the case $F(t,A)=x^2-A$. The result of Vinogradov on the distribution of $k$-power residues corresponds to the case $F(t,A)=x^k-A$ ($a$ is a $k$-power residue iff $x^k-a$). The results of Shparlinski cited above correspond to $F(t,\A)=t^n+A_1t^{n-1}+\ldots+A_n$ and $F(t,A_1,A_2)=t^n+A_1t+A_2$, however Theorem \ref{thm1} implies (\ref{shpa1}) with the weaker error term $O_n\lb p^{n-1/2}\log^n p\rb$ and (\ref{shpa2}) only under the stronger assumption $H_1H_2>p^{3/2+\eps}$. The stronger error terms in (\ref{shpa1}) and (\ref{shpa2}) were obtained using Burgess's estimates of character sums in short intervals and Deligne's estimates of the number of points on surfaces, however the special shape of the considered polynomials was exploited in the application of these methods and it is not clear how to obtain such an improvement in the general case.

{\bf Remark 2.} The condition $p>d$ in Theorem \ref{thm1} generally cannot be dropped. For example if we fix $p$ and take $F(t,A)=t^p-t-A$ then we have $G\cong\Z/p$ (by Artin-Schreier theory) and $\nu=1$, so Theorem \ref{thm1} would predict that the number of $a\in S\subset\F_q$ with $C(F(t,a))=\{1\}$ is $|S|/p+O_p(q^{-1/2}\irreg(S))$. However if we take $S=\{a\in\F_q|\tr_{\F_q/\F_p}(a)=0\}$ it turns out that $F(t,a)$ splits completely (i.e. $C(F(t,a))=\{1\}$) for all $a\in S$, while $\irreg(S)=p$ remains bounded, so the conclusion of Theorem \ref{thm1} does not hold in this case.

{\bf Acknowledgment.} The author would like to thank Ze\'ev Rudnick for a discussion which lead the author to the present work. The author would also like to thank Ze\'ev Rudnick and Igor Shparlinski for many valuable comments on earlier drafts of this paper.

\section{Applications}

Before proving our main result we will present a couple of interesting consequences. The first one concerns the distribution of irreducible polynomials in short intervals in $\F_p[t]$ of the form $\{f+a|a\in\F_p\}$. Kurlberg and Rosenzweig \cite{KuRo18_} have shown that if $f\in\F_p[t]$ is a Morse polynomial of degree $d$ (i.e. $f$ takes $d-1$ distinct values at the zeros of its derivative) and $h_1,\ldots,h_m\in\F_p$ are distinct then 
\begin{multline*}\abs{\{a\in\F_p|f+a+h_1,\ldots,f+a+h_m\mbox{ are irreducible}\}}=\\=\frac{p}{d^m}\lb 1+O_{d,m}\lb p^{-1/2})\rb\rb.\end{multline*}
We extend this result to even shorter intervals. We mention that the following corollary was recently obtained independently by Kurlberg and Rosenzweig.
\begin{cor} Let $f\in\F_p[t],\deg f=d$ be a Morse polynomials and $h_1,\ldots,h_m$ distinct elements. Let $I=\{\be,\ldots,\be+H-1\}\ss\F_p$ be an interval of length $H$. Then
\begin{multline*}\abs{\{a\in I|f+a+h_1,\ldots,f+a+h_k\mbox{ are irreducible}\}}=\\=\frac{H}{d^m}\lb 1+O_{d,m}\lb \frac{p^{1/2}\log p}{H}\rb\rb.\end{multline*}
\end{cor}

\begin{proof} Consider $F(t,A)=\prod_{i=1}^m(f(t)+h_i+A)$. For $a\in\F_p$ such that $(\Disc_t(F))(a)\neq 0$ the polynomials $f+a+h_i$ are irreducible iff $C(F(t,a))$ has cycle structure $(d,d,\ldots,d)$. It is shown in \cite{GrKu08}*{Proposition 18} that for $d,m$ fixed and $p$ large enough the Galois group of $F(t,A)$ over $\F_p(A)$ is $G\cong S_d^m$. If $C$ is the class of products of $d$-cycles in $G$ then
$|C|=\frac{1}{d^m}|G|$ and now the corollary follows from Theorem \ref{thm1} applied to $F(t,A)$ and the set $S=I$ and (\ref{irregint}).\end{proof}

Our next application is to the distribution of irreducible values of a bivariate polynomial $F(t,x)\in\F_q[t][x]$ ($q$ is a power of a prime $p$) when we substitute for $x$ a polynomial $f=a_0+\ldots+a_nt^n$ with coefficients $(a_0,\ldots,a_n)$ taken in a subset $S\ss\F_q^{n+1}$. The case $S=\F_q^{n+1}$ was treated in \cite{Ent18_}*{Corollary 1.1}, it is the function field analogue of the classical Bateman-Horn conjecture in the $q\to\ity$ limit.

\begin{cor} Let $F_1,\ldots,F_m\in\F_q[t,x]\sm\F_q[t,x^p]$ be absolutely irreducible polynomials and denote $d_i=\deg_t F_i\lb t,A_0+A_1t+\ldots+A_nt^n\rb$ ($A_0,\ldots,A_n$ are free variables). Let $n$ be a natural number such that one of the following holds:
\begin{enumerate}
\item $n\ge 3$.
\item $n\ge 2$ and $p>2$.
\item $n\ge 1$ and $p>\max d_i$.
\end{enumerate}
Let $S\ss\{f\in\F_q[t]|\deg f\le n\}$ be a subset. Then
\begin{multline*}\abs{\{f\in S|F_i(t,f)\mbox{ are irreducible for }1\le i\le m\}}=\\=
\frac{|S|}{\prod_{i=1}^m d_i}
\lb 1+O_{m,n,d_i}\lb q^{-1/2}\cdot\irreg(S)\rb\rb\end{multline*}
(we define $\irreg(S)$ be viewing $S$ as a subset of $\F_q^{n+1}$).
\end{cor}

\begin{proof} Consider the polynomial
$$F(t,\A)=\prod_{i=1}^mF_i\lb t,A_0+A_1t+\ldots+A_nt^n\rb.$$ We have $d=\deg_tF=\sum_{i=1}^md_i$. It was shown in \cite{Ent18_} that under the conditions of the corollary the Galois group of $F$ is
$G=\prod_{i=1}^mS_{d_i}\ss S_d$. For $f=\sum_{i=0}^na_it^i$ we have that $F_i(t,f)$ are all irreducible iff $F(t,a_0,\ldots,a_n)$ factors into $m$ irreducible polynomials of degrees $d_1,\ldots,d_m$. Now the corollary follows from Theorem \ref{thm1} and the fact that the number of elements in $G$ with cycle structure $(d_1,\ldots,d_m)$ is $|G|/\prod_{i=1}^m d_i$.
\end{proof}

\section{Preliminaries}

Throughout this section $p$ is a prime, $q=p^k$ is a power of $p$ and $\F_q$ is the field with $q$ elements. We will denote by
\begin{equation}\label{sac}\psi(u)=e^{2\pi i\frac{\tr_{\F_q/\F_p}(u)}{p}}\end{equation} the standard additive character on $\F_q$.
\subsection{Additive character sums on curves}

Let $W/\F_q$ an absolutely irreducible projective algebraic curve. Let $f\in\F_q(W)$ be a rational function on $W$ defined over $\F_q$.

\begin{prop}\label{charsum}Assume that $(\deg f,p)=1$ and let $0\neq b\in\F_q$ be a nonzero element. Consider the additive character sum
$$S(f;b)=\sum_{x\in W(\F_q)\sm\{\mbox{poles of }f\}}\psi(bf(x)).$$
Then we have the estimate
$$|S(f;b)|\le\lb 2g(W)-2+2\deg f\rb q^{1/2},$$
where $g(W)$ is the genus of $W$.
\end{prop}

\begin{proof} The proposition follows from \cite{Mor91}*{Theorem 4.47(iii)}. The condition\\ $(\deg f,p)=1$ ensures that the extension of $\fqb(W)$ defined by the equation $y^p-y=f$ is a proper Artin-Schreier extension of degree $p$, so the conditions of \cite{Mor91}*{Theorem 4.47(iii)} are satisfied.\end{proof}

\subsection{Curve parametrizing specializations of a bivariate polynomial with prescribed factorization type}

Let $F(t,A)\in\F_q[t,A]$ be a squarefree polynomial in two variables without any irreducible factors lying in $\F_q[A,t^p]$ (equivalently $\Disc_t F\neq 0$) and let $d_1,\ldots,d_r$ be natural numbers with $\sum d_i=d=\deg_t F(t,A)$. We would like to determine for which $a\in\F_q$ the univariate polynomial $F(t,a)\in\F_q[t]$ decomposes into distinct irreducible polynomials of degrees $d_1,\ldots,d_r$. It turns out that these are parametrized by a (possibly reducible) affine algebraic curve defined over $\F_q$. The parametrization will not be one-to-one but with a fixed number of points on the curve corresponding to each $a$ as above.

First we briefly recall the notion and basic properties of Frobenius classes in Galois extensions of function fields, referring the reader to \cite{FrJa08}*{\S 6} or \cite{Ros02}*{\S 9} for a detailed exposition. Consider the extension field $L$ obtained by adjoining to $\F_q(A)$ the roots of $F(t,A)\in\F_q(A)[t]$. By our assumptions this extension is separable and therefore Galois. If $\F_{q^\nu}$ is the algebraic closure of $\F_q$ in $L$ then $L$ is the function field of a curve $W$ defined over $\F_{q^\nu}$. There is a surjective restriction map
\begin{equation}\label{galproj}G:=\Gal(L/\F_q(A))\twoheadrightarrow \Gal(\F_{q^\nu}/\F_q)\cong \Z/\nu.\end{equation} Consider the open subset $V\ss\A^1$ defined by 
$$V=\{a\in\fqb|\lb\Disc_t F\rb(a)\neq 0\}.$$
 A specialization $F(t,a)$ with $a\in\F_q$ is squarefree iff $a\in V(\F_q)$. With each $a\in V(\F_q)$ one can associate a conjugacy class in $G$ 
called the Frobenius class (also known as the Artin symbol) of $a$ in the extension $L/\F_q(A)$. We denote this class by $\Fr(a;L/\F_q(A))$. Any $\sig\in\Fr(a;L/\F_q(A))$ is mapped to the $q$-Frobenius $\Fr_q$ by the map in (\ref{galproj}). If we view $G$ as a subgroup of $S_d$ via its action on the roots of $F(t,A)\in\F_q(A)[t]$ then the cycle structure of $\Fr(a;L/\F_q(A))$ corresponds to the cycle structure of $\Fr_q$ acting on the roots of $F(t,a)$ in $\fqb$. Consequently $F(t,a)$ decomposes into distinct irreducible polynomials of degree $d_1,\ldots,d_r$ iff the cycle structure of $\Fr(a;L/\F_q(A))$ is $(d_1,\ldots,d_r)$. 

\begin{prop}\label{param1} Let $C$ be a conjugacy class of $G$ that is mapped to $\Fr_q\in\Gal(\F_{q^\nu}/\F_q)$ under the map in (\ref{galproj}). there exists an affine irreducible curve $W_C$ defined over $\F_q$ with a finite \'etale covering $W_C\to V$ defined over $\F_q$ such that for any point $a\in V(\F_q)$ the following are equivalent:
\begin{enumerate}
\item[(i)] $\Fr(a;L/\F_q(A))=C$.
\item[(ii)] There exists a point $z\in W_C(\F_q)$ lying over $a$.
\item[(iii)] There exist exactly $|G|/\nu|C|$ points in $W_C(\F_q)$ lying over $a$.
\end{enumerate}
Furthermore over $\F_{q^\nu}$ the curve $W_C$ becomes isomorphic to an open subset of the curve $W/\F_{q^\nu}$ with the function field $\F_{q^\nu}(W)=L$ defined above.
\end{prop}

\begin{proof} This is a fairly standard construction used most notably in one of the proofs of the function field version of the Chebotarev density theorem. A more general version of this construction applicable to any finite \'etale map of smooth varieties over $\F_q$ appears in the proof of \cite{Ent18_}*{Theorem 3}.\end{proof}

\begin{prop}\label{charsumdec} In the above setting assume additionally that $p>d$ (recall that $d=\sum d_i=\deg_t F$) and denote by $D(d_1,\ldots,d_r)$ the set of polynomials in $\F_q[t]$ that factor into distinct irreducible factors of degree $d_1,\ldots,d_r$. Let $0\neq b\in\F_q$ be a nonzero element. Then
$$\sum_{a\in\F_q\atop F(t,a)\in D(d_1,\ldots,d_r)}\psi(ba)=O_{\deg F}\lb q^{1/2}\rb,$$ here the implicit constant depends only on $\deg F$, the total degree of $F(t,A)$ in both variables and $\psi$ is the standard additive character defined in (\ref{sac}).
\end{prop}

\begin{proof} Let $C_1,\ldots,C_m$ be the conjugacy classes in $G=\Gal(L/\F_q(A))$ which when viewed as classes of permutations of the roots of $F(t,A)\in\F_q(A)[t]$ have cycle structure $(d_1,\ldots,d_r)$. Then $D(d_1,\ldots,d_r)$ is the disjoint union of the sets $$\{a\in V(\F_q)|\Fr(a;L/\F_q(A))=C_i\}.$$ 
For each $C_i$ consider the curve $W_{C_i}$ defined in Proposition \ref{param1} and denote by $f_i:W_{C_i}\to V$ the associated \'etale cover. Using Proposition \ref{param1} and Proposition \ref{charsum} we obtain
\begin{multline*}\sum_{a\in\F_q\atop F(a,t)\in D(d_1,\ldots,d_r)}\psi(ba)
=\sum_{i=1}^m\sum_{a\in V(\F_q)\atop{\Fr(a;L/\F_q(A))=C_i}}\psi(ba)=\\=\sum_{i=1}^m\frac{\nu|C_i|}{|G|}\sum_{z\in W_{C_i}\atop{f_i(z)\in V}}\psi(bf_i(z))=O(mS(f_i;b))=\\=O(m(2g(W_{C_i})-2+2\max\deg f_i)q^{1/2})=O_{\deg F}(q^{1/2}).\end{multline*}
Here we used the fact that $p>d$, so $(d!,p)=1$ and therefore $([L:\F_q(A)],p)=1$, so $(\deg f_i,p)=1$ for all $i$ since $\deg f_i|[L:\F_q(t)]$. We also used the estimates $g(W_{C_i})=O_{\deg F}(1)$ (follows from the Riemann-Hurwitz formula) and $$|\F_q\sm V(\F_q)|\le\deg_A\Disc_t F=O_{\deg F}(1).$$
\end{proof}

\subsection{Complexity of quasiprojective varieties and an explicit Chebotarev density theorem for varieties over a finite field}

Let $X$ be a quasiprojective variety defined over an algebraically closed field. We define the complexity of $X$ to be the minimal natural number $c$ such that $X$ is isomorphic to the locally closed subset of $\P^n$ defined by $$F_1(\mathbf{X})=\ldots=F_k(\mathbf{X})=0,G_1(\mathbf{X})\cdots G_m(\mathbf{X})\neq 0,$$ where $F_i,G_j$ are homogeneous polynomials in the variables $\mathbf{X}=(x_0,\ldots,x_n)$ and $n,\deg F_i,\deg G_j<c$. We denote the complexity of $X$ by $\comp(X)$. If $\phi:X\to Y$ is a morphism of quasiprojective varieties we define the complexity of $\phi$ to be the minimal natural number $c$ such that the graph of $\phi$ embedded in projective space via a Segre embedding of $X\times Y$ can be defined by equations and inequalities of degree at most $c$ in at most $c$ variables. See \cite{Ent18_}*{\S 4} and \cite{BGT11}*{Appendix A} for the basic properties of this notion of complexity.

Now let $\phi:X\to Y$ be a finite \'etale morphism of varieties defined over a finite field $\F_q$ with $Y$ normal and irreducible. Denote $K=\F_q(Y)$ (the field of rational functions on $Y$ defined over $\F_q$) and let $L$ be the Galois closure of the compositum of $\F_q(X_i)$ over $K$, where $X_i$ are the $\F_q$-irreducible components of $X$. The group $G=\Gal(L/K)$ is isomorphic to the \'etale monodromy group of the map $\phi$. With each $y\in Y(\F_q)$ one can associate a unique conjugacy class $\Fr(y;\phi)\ss G$ called the Frobenius class of $y$ with respect to the map $\phi$. This generalizes the notion of Frobenius class for a covering map of curves mentioned above and it is defined using the functoriality of the \'etale fundamental group. See \cite{Ent18_}*{\S 4} for its precise definition and basic properties.

There is an exact sequence
\begin{equation}\label{agmon}1\to\Gal(L\fqb/K\fqb)\to G\xrightarrow{\pi}\Gal(\F_{q^\nu}/\F_q)\to 1,\end{equation} where $\F_{q^\nu}$ is the algebraic closure of $\F_q$ in $L$ and for every $y\in Y(\F_q)$ we have $\pi(\Fr(y;\phi))=\Fr_q$.

Now we state the explicit Chebotarev density theorem for varieties over a finite field. There are several versions of this theorem and we cite the one appearing in \cite{Ent18_}. A very similar version appears in \cite{ABR15}*{Theorem A.4} (stated in the language of rings instead of varieties) and another slightly weaker version can be found in \cite{Cha97}*{Theorem 4.1}.

\begin{thm}[\cite{Ent18_}*{Theorem 3}]\label{cdt} Let $\phi:X\to Y$ and $G$ be as above and let $C$ be a conjugacy class of $G$ such that $\pi(C)=\{\Fr_q\}$. Then
$$\abs{\{y\in Y(\F_q)|\Fr(y;\phi)=C\}}=\nu\frac{|C|}{|G|}q^n\lb 1+O_{\comp(\phi)}(q^{-1/2})\rb.$$\end{thm}

\begin{cor}\label{corcheb}With notation as in the introduction let $F\in\F_q[t,\A]$ be a squarefree polynomial without irreducible factors in $\F_q(t^p,A)$ and let $G$ be the Galois group of $F$ over $\F_q(\A)$, which we view as a subgroup of $S_d$. Let $d_1,\ldots,d_r$ be natural numbers with $\sum d_i=d$ and $C\ss S_d$ the conjugacy class of permutations with cycle structure $(d_1,\ldots,d_r)$. Denote by $D(d_1,\ldots,d_r)$ the set of $\a\in\F_q^n$ such that $F(t,\a)\in\F_q[t]$ decomposes into distinct irreducible factors of degree $d_1,\ldots,d_r$. Then
$$\abs{D(d_1,\ldots,d_r)}=\nu\frac{|C\cap\pi^{-1}(\Fr_q)|}{|G|}q^n\lb 1+O_{n,\deg F}(q^{-1/2})\rb,$$
where $\F_{q^\nu}$ is the algebraic closure of $\F_q$ in the extension of $\F_q(\A)$ obtained by adjoining the roots of $F(t,\A)\in\F_q(\A)[t]$ and $\pi:G\to\Gal(\F_{q^\nu}/\F_q)$ is the restriction map.
\end{cor}

\begin{proof} Let $Y$ be the open subset in the affine $n$-space defined by $(\Disc_t F)(\A)\neq 0$ (by our assumption on $F$ we have $\Disc_t F\neq 0$, so this open set is nonempty). Consider the variety
$X=\{(\a,\tau)\ss Y\times\A^1: F(\tau,\a)=0\}$ and the projection map $\phi(\a,\tau)=\a$. The variety $Y$ is smooth and irreducible and $\phi:X\to Y$ is finite \'etale of degree $d$ (since $(\Disc_t F)(\a)\neq 0$ for all $\a\in Y$). Therefore the conditions of Theorem \ref{cdt} hold with $G$ and $\nu$ as in the statement of the corollary. 

Recalling that $C\ss S_d$ is the conjugacy class of permutations with cycle structure $d_1,\ldots,d_r$ we have $\a\in D(d_1,\ldots,d_r)$ iff $\Fr(a;\phi)\ss C$. Now writing $C\cap\pi^{-1}(\Fr_q)$ as the disjoint union of $G$-conjugacy classes $C_1,\ldots,C_m$ and applying Theorem \ref{cdt} we obtain
$$\abs{D(d_1,\ldots,d_r)}=\nu\frac{|C\cap\pi^{-1}(\Fr_q)|}{|G|}q^n\lb 1+O_{\comp(\phi)}(q^{-1/2})\rb.$$
It remains to note that $\comp(\phi)=O_{n,\deg F}(1)$, since the affine part of the graph of $\phi$ is $$\{(\a,\tau,\a):F(\tau,\a)=0,(\Disc_t F)(\a)\neq 0\}.$$
\end{proof}

\section{Proof of Theorem \ref{thm1}}

Let $F\in\F_q[t,\A]$ be a squarefree polynomial without irreducible factors lying in $\F_q[t^p,\A]$, or equivalently $\Disc_t F\neq 0$. We will denote $\Del(\A)=\Disc_t F$. Denote $d=\deg_t F$ and let $d_1,\ldots,d_r$ be natural numbers with $\sum d_i=d$. We denote by $D(d_1,\ldots,d_r)$ the set of $\a\in\F_q^n$ such that $F(t,\a)$ factors into $r$ distinct irreducible factors of degree $d_1,\ldots,d_r$. Let $S\ss\F_q^n$ be a subset. Note that
$$\bigcup_{d_1,\ldots,d_r\atop{\sum d_i=d}}D(d_1,\ldots,d_r)=\{\a\in\F_q^n|\Del(\a)\neq 0\}$$
($F(t,\a)$ is squarefree iff $\Del(\a)\neq 0$).
We would like to estimate the size of $S\cap D(d_1,\ldots,d_r)$.

\begin{lem}\label{lem1} $$\abs{S\cap D(d_1,\ldots,d_r)}=\frac{|S|\cdot|D(d_1,\ldots,d_r)|}{q^n}+
q^n\sum_{0\neq \b\in\F_q^n}\hat{\one_S}(\b)\sum_{\a\in D(d_1,\ldots,d_r)}\psi(-\a\cdot\b).$$\end{lem}

\begin{proof} By Plancherel's identity and (\ref{fourier}) we have
\begin{multline*}\abs{S\cap D(d_1,\ldots,d_r)}=\sum_{\a\in\F_q^n}\one_S(\a)\one_{D(d_1,\ldots,d_r)}(\a)=\\=q^n\sum_{\b\in\F_q^n}\hat{\one_S}(\b)\hat{\one_{D(d_1,\ldots,d_r)}}(\b)=\\
=\frac{|S|\cdot |D(d_1,\ldots,d_r)|}{q^n}+
\sum_{0\neq \b\in\F_q^n}\hat{\one_S}(\b)\sum_{\a\in D(d_1,\ldots,d_r)}\psi(-\a\cdot \b),\end{multline*}
here we also used the fact that $$\hat{\one_S}(0)=q^{-n}|S|,\hat{\one_{D(d_1,\ldots,d_r)}}(0)=q^{-n}|D(d_1,\ldots,d_r)|.$$
\end{proof}

\begin{prop}\label{propkey}For $0\neq\b\in\F_q^n$ we have
$$\sum_{\a\in D(d_1,\ldots,d_r)}\psi(-\a\cdot\b)=O_{\deg F}\lb q^{n-1/2}\rb.$$\end{prop}

\begin{proof} Write $\b=(b_1,\ldots,b_n)$. Assume without loss of generality that $b_1\neq 0$.
Let $a_2,\ldots,a_n\in\F_q$ be any elements. If $\Del(A,a_2,\ldots,a_n)=0$ ($A$ is a variable) or \\ $\deg_tF(t,A,a_2,\ldots,a_n)<d$ then for any $a_1\in\F_q$ we have $(a_1,a_2,\ldots,a_n)\not\in D(d_1,\ldots,d_r)$.
If $\Del(A,a_2,\ldots,a_n)\neq 0$ and $\deg_tF(t,A,a_2,\ldots,a_n)=d$ then we may apply Proposition \ref{charsumdec} to the polynomial $F(t,A,a_2,\ldots,a_n)\in\F_q(t,A)$ and obtain
$$\sum_{a_1\in\F_q\atop{(a_1,a_2,\ldots,a_n)\in D(d_1,\ldots,d_r)}}\psi(-a_1b_1)=O_{\deg F}\lb q^{1/2}\rb.$$ Therefore
\begin{multline*}\sum_{\a\in D(d_1,\ldots,d_r)}\psi(-\a\cdot\b)=\\=
\sum_{(a_2,\ldots,a_n)\in\F_q^{n-1}}\psi(-a_2b_2-\ldots-a_nb_n)\sum_{a_1\in\F_q\atop{(a_1,\ldots,a_n)\in D(d_1,\ldots,d_r)}}\psi(-a_1b_1)=\\=O_{\deg F}\lb q^{n-1/2}\rb.\end{multline*}
\end{proof}

Now we can complete the proof of Theorem \ref{thm1}. By the definition of $\irreg(S)$ we have
$$\sum_{\b\in\F_q^n}\abs{\hat{\one_S}(\b)}=\frac{|S|}{q^n}\irreg(S)$$ and so
by Proposition \ref{propkey} we have
\begin{multline*}\abs{\sum_{0\neq \b\in\F_q^n}\hat{\one_S}(\b)\sum_{\a\in D(d_1,\ldots,d_r)}\psi(-\a\cdot\b)}\le\sum_{0\neq\b\in\F_q^n}\abs{\hat{\one_S}(b)}\cdot O_{\deg F}\lb q^{n-1/2}\rb=\\
=O_{\deg F}\lb q^{-1/2}|S|\cdot\irreg(S)\rb\end{multline*} and now Theorem \ref{thm1} immediately follows from
Lemma \ref{lem1} and Corollary \ref{corcheb}.

\section{Appendix: the irregularity of a product of intervals}

In this appendix we derive the bound (\ref{irregint}) for the irregularity of a product $S=I_1\times\ldots\times I_n\ss\F_p^n$ ($p$ is a prime) where $I_1,\ldots,I_n\ss\F_p$ are intervals or more generally arithmetic progressions, i.e. $$I_i=\{\al_ik+\be_i: k=0,1,\ldots,H_i-1\}$$ with $\al_i,\be_i\in\F_p, \al_i\neq 0$ and $1\le H_i\le p$.

\begin{lem}\label{lema1} For any subsets $S_i\ss\F_p$ and $S=S_1\times\ldots\times S_n\ss\F_p^n$ we have $$\irreg(S_1\times\ldots\times S_n)=\prod_{i=1}^n\irreg(S_i).$$\end{lem}

\begin{proof}This follows from (\ref{irreg}) and the fact that
$\hat{\one_S}(b_1,\ldots,b_n)=\prod_{i=1}^n\hat{\one_{S_i}}(b_i).$\end{proof}

\begin{lem}\label{lema2}For any subset $T\ss\F_p$ and $\al,\be\in\F_p,\al\neq 0$ if we denote $T'=\{\al a+\be|a\in T\}$ then $\irreg(T')=\irreg(T)$.\end{lem}

\begin{proof}We have \begin{multline*}\hat{\one_{T'}}(b)=\sum_{a\in T'}e^{-2\pi i ab/p}=
\sum_{a\in T}e^{-2\pi i(\al a+\be)b/p}=\\=e^{-2\pi i \be b/p}\sum_{a\in T}e^{-2\pi i a(\al b)/p}=
e^{-2\pi i \be b/p}\cdot\hat{\one_{T}}(\al b)\end{multline*} and therefore
$\abs{\hat{\one_{T'}}(b)}=\abs{\hat{\one_{T}}(\al b)}$. Since $\al\neq 0$ we have
$$\irreg(T')=\frac{p}{|T'|}\sum_{b\in\F_p}\abs{\hat{\one_{T'}}(b)}=
\frac{p}{|T|}\sum_{b\in\F_p}\abs{\hat{\one_{T}}(\al b)}=\irreg(T).$$
\end{proof}

Consequently it is enough to estimate $\irreg(I)$ for an interval of the form $I=\{0,1,\ldots,H-1\}\ss\F_p$.
We have $\hat{\one_I}(0)=H/p$ and for $b\neq 0$
\begin{equation}\label{e1}\abs{\hat{\one_I}(b)}=\frac{1}{p}\abs{\sum_{k=0}^{H-1}e^{-2\pi ib}}=\frac{1}{p}\abs{\frac{e^{-2\pi i bH}-1}{e^{-2\pi i b}-1}}\le \frac{2}{p\abs{\sin\pi b/p}}.\end{equation}
We can identify $b$ with a nonzero integer in the interval $[-(p-1)/2,(p-1)/2]$ (we may assume that $p>2$). Then by (\ref{e1}) and the inequality $|\sin \pi t|\ge 2t$ valid for $0<|t|<1/2$ we have
$$\sum_{0\neq b\in\F_p}\abs{\hat{\one_I}(b)}\le 8\sum_{b=1}^{(p-1)/2}\frac{1}{b}\le 8\log p$$
and so $$\irreg(I)\le 1+\frac{p}{H}\cdot 8\log p\le \frac{9p\log p}{H}.$$

Now by Lemma \ref{lema1} and Lemma \ref{lema2} for $S=I_1\times\ldots\times I_n$ we have
$$\irreg(S)\le\prod_{i=1}^n \frac{9p\log p}{H_i},$$ which is precisely (\ref{irregint}).

\bibliography{../Tex/mybib}

\begin{bibdiv}
\begin{biblist}

\bib{ABR15}{article}{
      author={Andrade, J.~C.},
      author={Bary-Soroker, L.},
      author={Rudnick, Z.},
       title={Shifted convolution and the {Titchmarsh} divisor problem over
  $\mathbb{F}_q[t]$},
        date={2015},
     journal={Philos. Trans.},
      volume={373},
      number={2040},
}

\bib{BBF18}{article}{
      author={Bank, E.},
      author={Bary-Soroker, L.},
      author={Fehm, A.},
       title={Sums of two squares in short intervals in polynomial rings over
  finite fields},
        date={2018},
     journal={Amer. J. Math.},
      volume={140},
      number={4},
       pages={1113\ndash 1131},
}

\bib{BGT11}{article}{
      author={Breuillard, E.},
      author={Green, B.},
      author={Tao, T.},
       title={Approximate subgroups of linear groups},
        date={2011},
     journal={Geom. and Funct. Anal.},
      volume={21},
       pages={774\ndash 819},
}

\bib{Bur57}{article}{
      author={Burgess, D.~A.},
       title={The distribution of quadratic residues and nonresidues},
        date={1957},
     journal={Mathematika},
      volume={4},
       pages={106\ndash 112},
}

\bib{Cha97}{article}{
      author={Chavdarov, N.},
       title={The generic irreducibility of the numerator of the zeta function
  in a family of curves with large monodromy},
        date={1997},
     journal={Duke Math. J.},
      volume={87},
      number={1},
       pages={151\ndash 180},
}

\bib{Ent18_}{article}{
      author={Entin, A.},
       title={Monodromy of hyperplane sections of curves and decomposition
  statistics over finite fields},
        date={2018},
     journal={ArXiv:1805.05454[math:NT]},
}

\bib{FrJa08}{book}{
      author={Fried, M.~D.},
      author={Jarden, M.},
       title={Field arithmetic},
     edition={third},
      series={Ser. Modern Surv. Math.},
   publisher={Springer-Verlag},
        date={2008},
      volume={11},
}

\bib{GrKu08}{article}{
      author={Granville, A.},
      author={Kurlberg, P.},
       title={Poisson statistics via the chinese remainder theorem},
        date={2008},
     journal={Adv. in Math.},
      volume={218},
      number={6},
       pages={2013 – 2042},
}

\bib{Kar16_}{article}{
      author={Karidi, T.},
       title={On {Galois} groups of compositions},
        date={2016},
     journal={M.Sc. Thesis, Tel Aviv University},
}

\bib{KuRo18_}{article}{
      author={Kurlberg, P.},
      author={Rosenzweig, L.},
       title={Prime and {M\"obius} correlations for very short intervals in
  $\mathbf{F}_p$},
        date={2018},
     journal={arXiv:1802.01215 [math.NT]},
}

\bib{Mor91}{book}{
      author={Moreno, C.~J.},
       title={Algebraic curves over finite fields},
   publisher={Cambridge University Press},
        date={1991},
}

\bib{Pol18}{article}{
      author={P{\'o}lya, G.},
       title={{\"U}ber die {V}erteilung der quadratischen eeste und
  nichtreste},
        date={1918},
     journal={G{\"o}ttinger Nachrichten},
       pages={21\ndash 29},
}

\bib{Ros02}{book}{
      author={Rosen, M.~I.},
       title={Number theory in function fields},
      series={GTM},
   publisher={Springer-Verlag},
        date={2002},
      volume={210},
}

\bib{Shp11}{article}{
      author={Shparlinski, I.~E.},
       title={On the distribution of irreducible trinomials},
        date={2011},
     journal={Canad. Math. Bull.},
      volume={54},
      number={4},
       pages={748\ndash 756},
}

\bib{Shp89}{article}{
      author={Shparlinski, I.~E.},
       title={On polynomials of prescribed height over finite fields},
        date={1989},
     journal={Math. USSR Sb.},
      volume={63},
      number={1},
       pages={247\ndash 255},
}

\bib{Shp92}{article}{
      author={Shparlinski, I.~E.},
       title={On primitive elements in finite fields and on elliptic curves},
        date={1992},
     journal={Math. USSR Sb.},
      volume={71},
      number={1},
       pages={41\ndash 50},
}

\bib{Vin18}{article}{
      author={Vinogradov, I.~M.},
       title={Sur la distribution des r\'esidus et des non-r\'esidus des
  puissances},
        date={1918},
     journal={J. Phys.-Math. Soc. Univ. Perm},
      volume={1},
       pages={94\ndash 96},
}

\bib{Vin27}{article}{
      author={Vinogradov, I.~M.},
       title={On a general theorem concerning the distribution of the residues
  and non-residues of powers},
        date={1927},
     journal={Trans. Amer. Math. Soc.},
      volume={29},
      number={1},
       pages={209\ndash 217},
}

\end{biblist}
\end{bibdiv}
\bibliographystyle{amsrefs}

\end{document}